\documentclass[a4paper,11pt,twoside]{amsart}

\usepackage{amssymb}
\usepackage{vmargin}
\usepackage{mathrsfs}
\usepackage[all]{xy}

\setmargins{32mm}{20mm}{14.6cm}{22cm}{1cm}{1cm}{1cm}{1cm}

\usepackage{xr}
\newcommand\da{\!\downarrow\!}

\newcommand\la{\leftarrow}

\newcommand\ten{\otimes}
\newcommand\hten{\hat{\otimes}}

\newcommand\CC{\mathrm{C}}

\renewcommand\H{\mathrm{H}}

\newcommand\HC{\mathrm{HC}}
\newcommand\HP{\mathrm{HP}}
\newcommand\HN{\mathrm{HN}}

\newcommand\Z{\mathbb{Z}}
\newcommand\Q{\mathbb{Q}}

\newcommand\R{\mathbb{R}}
\newcommand\Cx{\mathbb{C}}

\newcommand\bA{\mathbb{A}}

\newcommand\bG{\mathbb{G}}

\newcommand\bP{\mathbb{P}}

\newcommand\bS{\mathbb{S}}

\newcommand\C{\mathcal{C}}

\newcommand\cA{\mathcal{A}}
\newcommand\cB{\mathcal{B}}

\newcommand\cD{\mathcal{D}}

\newcommand\cJ{\mathcal{J}}

\newcommand\sA{\mathscr{A}}

\newcommand\sE{\mathscr{E}}

\newcommand\sH{\mathscr{H}}

\newcommand\sO{\mathscr{O}}

\renewcommand\L{\Lambda}

\newcommand\Alg{\mathrm{Alg}}
\newcommand\CAlg{\mathrm{CAlg}}

\newcommand\Hom{\mathrm{Hom}}

\newcommand\HHom{\underline{\mathrm{Hom}}}

\newcommand\Ext{\mathrm{Ext}}

\newcommand\cone{\mathrm{cone}}
\newcommand\cocone{\mathrm{cocone}}

\newcommand\per{\mathrm{per}}

\newcommand\Gal{\mathrm{Gal}}

\newcommand\Ch{\mathrm{Ch}}
\newcommand\ch{\mathrm{ch}}

\newcommand\Spec{\mathrm{Spec}\,}

\newcommand\Set{\mathrm{Set}}

\newcommand\Aff{\mathrm{Aff}}

\newcommand\Poly{\mathrm{Poly}}

\newcommand\FD{\mathrm{FD}}

\newcommand\Lim{\varprojlim}
\newcommand\LLim{\varinjlim}

\newcommand\ho{\mathrm{ho}\!}
\newcommand\into{\hookrightarrow}

\newcommand\abuts{\implies}
\newcommand\xra{\xrightarrow}

\newcommand\fin{\mathrm{fin}}

\newcommand\ssp{\mathrm{ssp}}
\newcommand\Blanc{\mathrm{Blanc}}

\newcommand\bt{\bullet}
\newcommand\by{\times}

\newcommand\Vect{\mathrm{Vect}}

\newcommand\Symm{\mathrm{Symm}}

\newcommand\an{\mathrm{an}}

\newcommand\Tot{\mathrm{Tot}\,}

\newcommand\ind{\mathrm{ind}}
\newcommand\pro{\mathrm{pro}}

\newcommand\Sm{\mathrm{Sm}}

\newcommand\red{\mathrm{red}}

\newcommand\dR{\mathrm{dR}}

\newcommand\op{\mathrm{opp}}

\newcommand\co{\colon\thinspace}

\newcommand\oR{\mathbf{R}}

\newcommand\oL{\mathbf{L}}

\newcommand\uleft\underleftarrow
\newcommand\uline\underline
\newcommand\uright\underrightarrow

\newtheorem{theorem}{Theorem}[section]
\newtheorem{proposition}[theorem]{Proposition}
\newtheorem{corollary}[theorem]{Corollary}

\newtheorem{lemma}[theorem]{Lemma}
\newtheorem*{theorem*}{Theorem}
\newtheorem*{proposition*}{Proposition}
\newtheorem*{corollary*}{Corollary}
\newtheorem*{lemma*}{Lemma}
\newtheorem*{conjecture*}{Conjecture}

\theoremstyle{definition}
\newtheorem{definition}[theorem]{Definition}

\newtheorem*{definition*}{Definition}

\theoremstyle{remark}

\newtheorem{remark}[theorem]{Remark}
\newtheorem{remarks}[theorem]{Remarks}

\newtheorem*{example*}{Example}
\newtheorem*{examples*}{Examples}
\newtheorem*{remark*}{Remark}
\newtheorem*{remarks*}{Remarks}
\newtheorem*{exercise*}{Exercise}
\newtheorem*{property*}{Property}
\newtheorem*{properties*}{Properties}

\begin{document}
\title{Smooth functions on algebraic $K$-theory}
\author{J.P.Pridham}
\thanks{This work was supported by  the Engineering and Physical Sciences Research Council [grant number EP/I004130/2].}

\begin{abstract}
 For any complex scheme $X$ or  any dg category, there is an associated $K$-theory presheaf on the category of complex affine schemes. We study real $\C^{\infty}$ functions on this presheaf, defined by Kan extension, and show that they are closely related to real Deligne cohomology. When $X$ is quasi-compact and semi-separated, and for various non-commutative derived schemes, these smooth  functions on $K$-theory are dual to the homotopy fibre of the Chern character from Blanc's semi-topological $K$-theory to  cyclic homology.
 \end{abstract}

\maketitle

\section*{Introduction}

 The idea of looking at smooth functions on algebraic $K$-theory presheaves was introduced in \cite{smoothK}. There, we looked at presheaves on a category of Fr\'echet manifolds, and showed that for the Fr\'echet algebra of complex analytic functions on a polydisc, the smooth functions on connective algebraic $K$-theory are the linear dual of a form of real Deligne cohomology.

The motivation behind this paper was the desire to adapt this characterisation to more general dg categories and to non-connective $K$-theory, in particular to obtain results for smooth proper complex varieties without having to impose descent by hand after the fact as in \cite{smoothK}. 

Rather than working with presheaves on Fr\'echet manifolds, our primary focus is to look at functors on $\C^{\infty}$-rings. This is justified because for any Fr\'echet manifold $M$ and 
proper dg category $\cA$,  we have $\C^{\infty}(M,\cA) \simeq \C^{\infty}(M,\R)\ten_{\R}\cA$, and $\C^{\infty}(M,\R) $ is a $\C^{\infty}$-ring.
Since we concentrate on dg categories $\cA$ without extra topological structure,  the adjunction between $\C^{\infty}$-rings and commutative $\Cx$-algebras then allows us to frame our constructions in terms of the $K$-theory presheaf 
\[
 \uline{K}(\cA)(Z):= K(\cA\ten_{\Cx}\Gamma(Z,\sO_Z))
\]
on complex affines  as considered by Blanc in \cite{blancToplKTh}. 

Our main result is Corollary \ref{maincor}, which shows that for any  quasi-compact and semi-separated derived scheme $X$, and for various non-commutative derived schemes as well, the smooth functions 
\[
 \oR\C^{\infty}(\uline{K}(X), \R)
\]
on the non-connective $K$-theory presheaf $\uline{K}(X)$ are given by $\R$-linear functions on the homotopy fibre of the Chern character
\[
\ch \co  K^{\Blanc,st}(X)_{\R} \to \HC^{\Cx}(X)_{[-2]},
\]
where $ K^{\Blanc,st}$ is Blanc's semi-topological $K$-theory. As explained in Remark \ref{Delignecoho}, this homotopy fibre only differs from real Deligne cohomology by a $\beta$-torsion module for the Bott element $\beta$. 

Our approach to proving this is very different from that in \cite{smoothK},   appealing directly to Goodwillie's comparison \cite{goodwillieChern} between $K$-theory and cyclic homology instead of resorting to symmetric space calculations. This appeal is made possible by considering the de Rham presheaf
\[
 \uline{K}(\cA)_{\dR}(B):=\uline{K}(\cA)(B^{\red}), 
\]
an idea which can be traced back to \cite{Gr,simpsonHtpy}. 
In \S \ref{blancsn}, we show 
that smooth functions on $\uline{K}(\cA)_{\dR}$ are just given by real linear functions on  $ K^{\Blanc,st}(\cA)$.  

de Rham  presheaves then  allow us to use  Goodwillie's comparison in a systematic fashion, and we reformulate it in Corollary \ref{goodwilliecor} to say that for many complex dg categories $\cA$ of geometric origin, the homotopy fibres of $ \uline{K}(\cA)_{\Q} \to \uline{K}(\cA)_{\dR,\Q}$ and $ \uline{\HC}^{\Q}(\cA) \to \uline{\HC}^{\Q}(\cA)_{\dR}$ are equivalent. The key calculations  of the paper are in \S \ref{eilmacsn}, showing that smooth functions on smash products of Eilenberg--MacLane spectra of complex affine spaces  are just given by real multilinear maps. As a consequence, Theorem \ref{HCthm} shows that smooth functions on the rational cyclic homology presheaf $ \uline{\HC}^{\Q}(\cA)$ are the same as linear functions on real cyclic homology $\HC^{\R}(\cA) $, with the main result Corollary \ref{maincor} as a consequence.

Finally, in \S \ref{vistas} we discuss various related constructions and generalisations of these results. These include  an analytic setting extending \cite{smoothK} by considering dg categories enriched in topological vector spaces. An  analogue of real Deligne cohomology is then given by   compactly supported distributions on the   algebraic $K$-theory presheaf on formal neighbourhoods of infinite-dimensional manifolds.  

\tableofcontents

\section{Smooth functions on algebraic $K$-theory}

\subsection{Preliminaries on  $\C^{\infty}$-rings}

\subsubsection{Complex affines}

\begin{definition}
 Write $\CAlg(\Cx)$ for the category of commutative $\Cx$-algebras. We denote the opposite category by $\Aff_{\Cx}$, the category of complex affine schemes. 
\end{definition}

\begin{definition}
 Let $\CAlg(\Cx)_{\fin} \subset \CAlg(\Cx)$ be the full subcategory  of finitely generated $\Cx$-algebras, and 
 write  $\Aff_{\Cx,\fin}:=  (\CAlg(\Cx)_{\fin})^{\op}$.
\end{definition}

Observe that because $\Cx$ is Noetherian, finitely generated commutative $\Cx$-algebras are finitely presented, so the colimit functor $\ind( \CAlg(\Cx)_{\fin})\to \CAlg(\Cx)$ is an equivalence.

\subsubsection{$\C^{\infty}$-rings}

We now consider the $\C^{\infty}$-rings of \cite[\S 2.2]{joyceAGCinfty}.

\begin{definition}
Recall that  a $\C^{\infty}$-ring on $A$ is a set equipped with compatible operations $A^n \to A$ for   every $\C^{\infty}$-morphism $f \co \R^n \to \R$. There is a forgetful functor from $\C^{\infty}$-rings to commutative $\R$-algebras.

 Write $\C^{\infty}\Alg$ for the category of  $\C^{\infty}$-rings, and denote the opposite category by $\C^{\infty}\Aff$. Given $Z \in \C^{\infty}\Aff$, denote the corresponding object of $\C^{\infty}\Alg $ by $O^{\infty}(Z)$; given $A \in \C^{\infty}\Alg$, denote the corresponding object of $\C^{\infty}\Aff $ by $\Spec_{\infty}A$.
\end{definition}

\begin{definition}
Let $\C^{\infty}\Alg_{\fin} \subset \C^{\infty}\Alg$ be the full subcategory on finitely presented objects, and write   $\C^{\infty}\Aff_{\fin}:=  (\C^{\infty}\Alg_{\fin})^{\op}$.
\end{definition}

Observe that  the colimit functor $\ind( \C^{\infty}\Alg_{\fin})\to \C^{\infty}\Alg$ is an equivalence.

\subsubsection{Smooth functions on complex affines}

\begin{lemma}\label{ustarlemma}
 The functor $u \co \C^{\infty}\Alg \to \CAlg(\Cx)$ given by $u(A)=A \ten_{\R} \Cx$ has a left adjoint $u^*$, which has the property that for any smooth $\Cx$-algebra $S$, $u^*S$ is the ring of real $\C^{\infty}$-functions on the complex manifold $(\Spec S)(\Cx)$. 
\end{lemma}
\begin{proof}
We can write $u$ as the composition of the forgetful functor $\C^{\infty}\Alg \to \CAlg (\R)$ with the tensor product functor $\ten_{\R} \Cx \co \C^{\infty}\Alg (\R) \to \C^{\infty}\Alg (\Cx)$.  

The categories  $\C^{\infty}\Alg$ and  $\CAlg(\R)$ are  monadic over the category of  real vector spaces, with  natural maps $\Symm_{\R}V \to \C^{\infty}(V^{\vee},\R)$ of the associated monads. Here, we regard $V^{\vee}$ as a pro-finite-dimensional vector space, so $\C^{\infty}(V^{\vee},\R)= \LLim_{\alpha}\C^{\infty}(V^{\vee}_{\alpha},\R) $ for  finite-dimensional subspaces $V_{\alpha}$ of $V$.
The forgetful functor thus has a left adjoint, which sends $\Symm_{\R}V$ to $ \C^{\infty}(V^{\vee},\R)$, and extends to all objects by passing to coequalisers.

The left adjoint to the functor $\ten_{\R}\Cx$ is given by Weil restriction of scalars, so sends $\Symm_{\Cx}U$ to $\Symm_{\R}U$ for any complex vector space $U$ (via either of the identifications $\Hom_{\R}(\Cx,\R) \cong \Cx$ or by $V\cong (V.\Gal(\Cx/\R))^{\Gal(\Cx/\R)}$). The left adjoint to the forgetful functor sends $\Symm_{\R}V$ to $ \C^{\infty}(V^{\vee},\R)$.

The composite $u^*$ thus satisfies $u^*\Cx[z_1, \ldots, z_n]= \C^{\infty}(\Cx^n,\R)$. If $S=\Cx[z_1, \ldots, z_n]/(f_1, \ldots, f_m)$, we have $u^*S=\C^{\infty}(\Cx^n,\R)/(\Re f_1,\Im f_1,  \ldots, \Re f_m, \Im f_m)$, which is precisely $\C^{\infty}((\Spec S)(\Cx),\R)$ when $S$ is smooth.
\end{proof}

The lemma motivates the following choice of notation:
\begin{definition}\label{cinftyfunctordef}
Given $Z= \Spec O(Z) \in  \Aff_{\Cx}$, write $\C^{\infty}(Z,\R):= u^*O(Z)$ and $Z_{\C^{\infty}}:= \Spec_{\infty} u^*O(Z)$.
\end{definition}
Beware that we cannot interpret $\C^{\infty}(Z,\R)$ as a ring of functions on $Z(\Cx)$ when $Z$ is not reduced. For instance, when $Z= \Spec \Cx[z]/z^2$ we have $\C^{\infty}(Z,\R)= \R[x,y]/(x^2-y^2, xy)$. 

\subsection{Presheaves and homotopy Kan extensions}

\begin{definition}
For any ring $k$, let $\Ch_k$ be the category of  (unbounded) chain complexes of $k$-modules. For any  category $I$, write $\Ch_k(I)$ for the category of presheaves in real chain complexes on $I$ (i.e. functors $I^{\op} \to \Ch_k$).
\end{definition}

As in \cite[Proposition \ref{smoothK-pshfmodelstr}]{smoothK}, the projective model structure of  \cite{bousfieldkan} adapts to give:
\begin{proposition}\label{pshfmodelstr}
 For any small category $I$, there is a  cofibrantly generated model structure on the category  $\Ch_k(I)$, with a morphism $f \co A \to B$ being a fibration (resp. weak equivalence) whenever the maps $f_i \co A(i) \to B(i)$ are surjections (resp. quasi-isomorphisms) for all $i \in I$.
\end{proposition}

\begin{definition}\label{NkXdef}
 For any object $X \in I$, we write $k.X \in \Ch_k(I)$ for the presheaf sending $Y$ to the free $k$-module $k.\Hom(Y,X)$ generated by the set $\Hom(Y,X)$.

Given a simplicial object $X \in I^{\Delta^{\op}}$, we write $Nk.X \in \Ch_k(I)$ for the Dold--Kan normalisation of the simplicial presheaf $k.X$ in $k$-modules.
\end{definition}

Note that $Nk.X$ is cofibrant because it is a direct summand of a complex $k.X_0 \la k.X_1 \la k.X_2 \la \ldots$ which can be constructed as a composition of pushouts of generating cofibrations $k.X_i[-i-1] \to \cone(k.X_i)[-i-1]$.

Since we want to consider presheaves on the categories $\Aff_{\Cx}$ and $\C^{\infty}\Aff$, we will now take the lazy, but notationally more convenient, route of assuming that $\Ch_k$ lives in a larger universe, so contains limits and colimits indexed by categories such as these.
The following lemmas show that  that assumption is unnecessary, and that all our presheaves can be defined on genuinely small categories:

\begin{lemma}\label{indlemma}
For any small category $I$, the  restriction functor
$ \iota^* \co \Ch_k(\pro(I)) \to \Ch_k(I)$ coming from the 
inclusion functor $\iota \co I \to \pro(I)$ has a left adjoint $\iota_!$. 
The functor $\iota_!$ gives an equivalence between $\Ch_k(I)$ and the full subcategory of $\Ch_k(\ind(I))$ consisting of presheaves $F$ which are locally of finite presentation in the sense that
\[
 \LLim_{\alpha}F(X_{\alpha}) \to F(\Lim_{\alpha}X_{\alpha}) 
\]
is an isomorphism for all filtered inverse systems $\{X_{\alpha}\}$ in $I$.
 \end{lemma}
\begin{proof}
Because $F \in \Ch_k(I)$ is given by $ \LLim_{Y \in (I \da F)} k.Y$, where the colimit is indexed by the comma category of morphisms $k.Y \to F$ in $\Ch_k(I)$, the functor $\iota_!$ can be defined by setting
\[
 \iota_!F = \LLim_{Y \in (I \da F)} k.\iota(Y).
\]

For a filtered inverse system $X= \{X_{\alpha}\}_{\alpha}$ in $\pro(I)$ and $Y \in I$, we necessarily have $\iota_!k.Y= k.(\iota Y)$, and so
\[
 (\iota_!k.Y)(X)= \LLim_{\alpha} (\iota_!k.Y)(X_{\alpha})= \LLim_{\alpha} k.Y(X_{\alpha}).
\]
We now just observe that any $F \in \Ch_k(I)$ can be written as a colimit of objects of the form $k.Y$, so the maps
\[
\LLim_{\alpha} F(X_{\alpha}) \la \LLim_{\alpha} (\iota_!F)(X_{\alpha})\to (\iota_!F)(X)
\]
 are isomorphisms.

Thus we have $(\iota_!F)( \{X_{\alpha}\}_{\alpha}) \cong \LLim_{\alpha} F(X_{\alpha})$, so $\iota^*\iota_!$ is the identity, while the counit $\iota_!\iota^*G \to G$ of the adjunction is an isomorphism if and only if $G$ is locally of finite presentation.
\end{proof}

\subsubsection{Presheaves on complex and $\C^{\infty}$-affines}
 
 Using  the equivalences $\pro(\Aff_{\Cx,\fin}) \simeq \Aff_{\Cx} $ and $\pro(\C^{\infty}\Aff_{\fin}) \simeq \C^{\infty}\Aff$, and noting that filtered colimits preserve quasi-isomorphisms of chain complexes, the lemma above gives:

\begin{lemma}\label{procatlemma}
 There are left Quillen functors
\begin{align*}
  \Ch_k(d\Aff_{\Cx,\fin})  \xra{ \iota_!} \Ch_k(\Aff_{\Cx})\\
\Ch_k(\C^{\infty}\Aff_{\fin})  \xra{\iota_!} \Ch_k(\C^{\infty}\Aff )
\end{align*}
whose essential images in the homotopy category consist of presheaves $F$ which are locally of finite presentation 
 in the sense that
\[
 \LLim_{\alpha}F(X_{\alpha}) \to F(\Lim_{\alpha}X_{\alpha}) 
\]
is a  quasi-isomorphism for all filtered inverse systems $\{X_{\alpha}\}$. The functor $\iota_!$ preserves weak equivalences and is full and faithful on the associated $\infty$-categories.
The right adjoint of $\iota_!$ is given by restriction $\iota^*$.
\end{lemma}

Now observe that the categories $\Aff_{\Cx,\fin} $ and $\C^{\infty}\Aff_{\fin}$ are equivalent to small categories, since each object has a finite set of generators and relations. We can therefore make sense of $\Ch_k(\Aff_{\Cx,\fin})$ and $\Ch_k(\C^{\infty}\Aff)$ without worrying about universes, as any construction we consider can be reduced to a construction in the original universe.

\begin{lemma}\label{rightadjointlemma}
For the functor $(-)_{\C^{\infty}}$ of Definition \ref{cinftyfunctordef}, and any $ F \in \Ch_k(\Aff_{\Cx,\fin})$, the presheaf $ \oL_!F_{\C^{\infty}} := \oL((-)_{\C^{\infty}})_!F \in\Ch_k(\C^{\infty}\Aff_{\fin}) $ is given (up to canonical weak equivalence) by 
\[
 \iota_!\oL_!F_{\C^{\infty}}\simeq U^* \iota_!F \in \Ch_k(\C^{\infty}\Aff),
\]
for the forgetful functor $U \co\C^{\infty}\Aff \to \Aff_{\Cx}$ given by $U\Spec_{\infty}A := \Spec (A\ten_{\R}\Cx)$.
\end{lemma}
\begin{proof}
 Since $(-)_{\C^{\infty}}$ is right adjoint to $U$ and $\oL U^*= U^*$, we have $U^*= \oL((-)_{\C^{\infty}})_! $. Because $\iota_!$ preserves weak equivalences, we have $\iota_!= \oL\iota_!$.  We then just note that  $(-)_{\C^{\infty}}$ restricts to a functor from $\Aff_{\Cx,\fin}$ to $\C^{\infty}\Aff_{\fin}$, and   commutes with $\iota$.
\end{proof}

\subsubsection{Smooth completions}\label{derivedcompletionsn}


\begin{definition}
Let $\pro(\FD\Vect_{\R})$ be the category of pro-finite-dimensional real vector spaces. This is equivalent to the opposite category $\Vect_{\R}^{\op}$ of the category of real vector spaces, by dualisation. We write $\Ch( \pro(\FD\Vect_{\R}))$ for the category of chain complexes in $ \pro(\FD\Vect_{\R})$, and give it the model structure coming from the equivalence $\Ch( \pro(\FD\Vect_{\R})) \simeq \Ch_{\R}^{\op}$.

Given $V= \{V_{\alpha}\}_{\alpha} \in \pro(\FD\Vect_{\R})$ and $U \in \Ch_{\R}$, we write $V\hten_{\R} U:= \Lim_{\alpha}(V_{\alpha}\ten_{\R} U)$; if $V^{\vee} \in \Ch_{\R}$ is the continuous dual of $V$, this means that $ V\hten_{\R} U$  is given by the $\Hom$-complex $\HHom_{\R}(V^{\vee},U)$ with $\HHom_{\R}(V^{\vee},U)_n= \prod_i\Hom_{\R}(V^{\vee}_i,U_{n-i}) $.
\end{definition}

\begin{definition}
Define $\Sm \co \Ch( \pro(\FD\Vect_{\R}))\to \Ch_{\Q}(\C^{\infty}\Aff)$ to be the functor given by 
\[
 \Sm(V)(Z):= V\hten_{\R}O^{\infty}(Z).
\]

We also write $\Sm$ for the functor $(-)_{\C^{\infty}}^*\Sm \co \Ch( \pro(\FD\Vect_{\R}))\to \Ch_{\Q}(\Aff_{\Cx})$ given by
\[
 \Sm(V)(Z):= V\hten_{\R}\C^{\infty}(Z,\R).
\]
\end{definition}

\begin{lemma}\label{Smlemma}
 The functors $\Sm \co \Ch( \pro(\FD\Vect_{\R}))\to \Ch_{\Q}(\C^{\infty}\Aff) $ and $\Sm \co \Ch( \pro(\FD\Vect_{\R}))\to \Ch_{\Q}(\Aff_{\Cx}^{}) $ are right Quillen.
\end{lemma}
\begin{proof}
 It is immediate that both functors preserve fibrations and trivial fibrations. Left adjoints 
$\Sm^*$
send presheaves $F$ to the coends
\begin{align*}
 Z &\mapsto \int^{Z \in \C^{\infty}\Aff} \HHom_{\Q}(F(Z), O^{\infty}(Z))^{\vee}. \\
Z &\mapsto \int^{Z \in \Aff_{\Cx}^{}} \HHom_{\Q}(F(Z), \C^{\infty}(Z,\R))^{\vee}. 
\end{align*}
\end{proof}
%
%

Lemmas \ref{rightadjointlemma} and \ref{procatlemma} imply that for $F \in \Ch_{\Q}(\Aff_{\Cx})$ locally of finite presentation, the complexes
\[
 \oL\Sm^*F, \quad  \oL\Sm^*(F|_{d\Aff_{\Cx,\fin}},),\quad \oL\Sm^*(F\circ U),\quad\text{and}\quad \oL\Sm^*((F\circ U)|_{d\C^{\infty}\Aff_{\fin}})
\]
are all equivalent. In particular, $ \oL\Sm^*F$ is determined by the functor $F(- \ten_{\R}\Cx)$ on  $\C^{\infty}$-rings.

\subsection{$K$-theory and cyclic homology  presheaves}

\subsubsection{$K$-theory presheaves}


\begin{definition}
Given a  spectrum $Y=\{Y^n\}$ and a ring $k$, write $Y_k$ for the chain complex
\[
 Y_k:=  \LLim_n N \bar{\CC}_{\bt}(Y^n,k))[n]
\]
of $k$-modules, where $\bar{\CC}$ denotes reduced chains and $N$ the Dold--Kan normalisation, and the maps
\[
 N\bar{\CC}_{\bt}(Y^n,k)\to N\bar{\CC}_{\bt}(Y^{n+1},k)[1]
\]
combine the structure map 
\[
 (N^{-1}k[-1])\ten_k \bar{\CC}_{\bt}(Y^n,k)\cong \bar{\CC}_{\bt}(S^1 \wedge Y^n,k)  \to\bar{\CC}_{\bt}(Y^{n+1},k)
\]
with the Eilenberg--Zilber shuffle map
\[
 k[-1]\ten_k N\bar{\CC}_{\bt}(Y^n,k) \to N( (N^{-1}k[-1])\ten_k \bar{\CC}_{\bt}(Y^n,k)). 
\]
\end{definition}
This construction is left adjoint to the Eilenberg--MacLane functor $H$ from $\Ch_k$ to spectra (cf. Definition \ref{Hdef}).

The following definition is taken from \cite{blancToplKTh}:
\begin{definition}
 Given a  dg category $\cA$ over $\Cx$, define the functor  $\uline{K}(\cA)$ from  $\CAlg_{\Cx}$ to spectra by 
\[
 \uline{K}(\cA)(B) := K(\cA\ten_{\Cx}B),
\]
 the non-connective $K$-theory of $\cA\ten_{\Cx}B$. 
Given a derived $\Cx$-scheme $X$, we write $ \uline{K}(X):= \uline{K}(\per_{dg}(X))$, 
for $\per_{dg}(X)$ the dg category of perfect complexes of $\sO_X$-modules on $X$.
\end{definition}

Combining these definitions, we obtain a presheaf $\uline{K}(\cA)_{\Q} \in  \Ch_{\Q}(d\Aff_{\Cx}) $ which is locally of finite presentation.

\begin{definition}\label{RCinftydef}
 Given a functor $F$ from $ \CAlg_{\Cx}$ to spectra, locally of finite presentation, define
\[
 \oR\C^{\infty}(F,\R):= (\oL\Sm^*(F_{\Q}))^{\vee},
\]
for $F_{\Q} \in \Ch_{\Q}(\Aff_{\Cx})$ as above. 

Given a functor $F$ from $ \CAlg_{\Cx}$ to simplicial sets, locally of finite presentation, define
\[
 \oR\C^{\infty}(F,\R):= \oR\C^{\infty}(\Sigma^{\infty}F_+,\R)=  (\oL\Sm^*(N\Q.F))^{\vee},
\]
for $N\Q.F$ as in Definition \ref{NkXdef}.

We refer to these as the complexes of  smooth functions on the presheaves $F$.
\end{definition}

\begin{definition}
 Given a functor $F$ from $ \C^{\infty}\Alg$ to spectra, locally of finite presentation, define
\[
 \oR O^{\infty}(F):= (\oL\Sm^*(F_{\Q}))^{\vee},
\]
for $F_{\Q} \in \Ch_{\Q}(\C^{\infty}\Aff)$ as above. 

Given a functor $F$ from $ \C^{\infty}\Alg$ to simplicial sets, locally of finite presentation, define
\[
 \oR O^{\infty}(F):= \oR O^{\infty}(\Sigma^{\infty}F_+)= (\oL\Sm^*(N\Q.F))^{\vee}.
\]

We refer to these as the complexes of  functions on the presheaves $F$.
\end{definition}

Equivalently, the dual functors $\oL O^{\infty}(-)^{\vee} \co \Ch_{\Q}(\C^{\infty}\Aff) \to \Ch(\pro(\FD\Vect_{\R}))$ and $\oL\C^{\infty}(-,\R)^{\vee}  \co \Ch_{\Q}(\Aff_{\Cx}) \to \Ch(\pro(\FD\Vect_{\R}))$  are the  derived enriched left Kan extensions of $  O^{\infty}(-)^{\vee}$  and $\C^{\infty}(-,\R)^{\vee} $ along the spectral Yoneda embeddings of $\C^{\infty}\Aff$ and $\Aff_{\Cx}$.

Our primary object of study will be $\oR\C^{\infty}(\uline{K}(\cA),\R)$, the complex of  smooth functions on the $K$-theory presheaf of $\cA$.

\subsubsection{Cyclic homology presheaves}

\begin{definition}
 Given a commutative ring $k$ and a dg category $\cA$ over $k$, write $\HC^k(\cA)$ for the chain complex associated to  cyclic homology  of $\cA$ over $k$. 
\end{definition}

\begin{definition}
 Given a  dg category $\cA$ over $\Cx$, define the functor  $\uline{HC}^{\Q}(\cA)\in \Ch_{\Q}(\Aff_{\Cx})$ by 
$\uline{HC}^{\Q}(\cA)(Z):= \HC^{\Q}(\cA\ten_{\Cx}O(Z))$.  Given a derived $\Cx$-scheme $X$, we write $ \uline{\HC}^{\Q}(X):= \uline{\HC}^{\Q}(\per_{dg}(X))$, for $\per_{dg}(X)$ the dg category of perfect complexes of $\sO_X$-modules on $X$.
\end{definition}

\section{de Rham presheaves and Blanc's semi-topological $K$-theory}\label{blancsn}

\begin{definition}
 Given $B \in \CAlg_{\Cx}$, define $B^{\red}$ to be the quotient of $B$ by its nilradical.
\end{definition}

The following definition is due to Carlos Simpson in \cite{simpsonHtpy}:

\begin{definition} Given a presheaf $F$ on $\Aff_{\Cx}$, define the de Rham presheaf  $F_{\dR}$ of $F$ by setting 
\[
 F_{\dR}(B):= F(B^{\red}).
\]
\end{definition}

\begin{proposition}\label{dRprop}
 For any affine $\Cx$-scheme $Z$ of finite type, the complex $\oR\C^{\infty}(Z_{\dR},\R)$ is naturally quasi-isomorphic to the derived global sections $\oR\Gamma(Z(\Cx)_{\an}, \R)$ of the space $Z(\Cx)$ with the analytic topology.
\end{proposition}
\begin{proof}
We first need to find a cofibrant replacement for the presheaf $(\Q.Z)_{\dR}$.  Choose a closed immersion $Z \into Y$ into a smooth $\Cx$-scheme $Y$, and write $\hat{Y}_Z$ for the formal completion of $Y$ along $Z$ --- this is an object of $\ind(\Aff_{\Cx, \fin})$. Now consider the simplicial ind-scheme $\tilde{Y}_Z$ given by $(\tilde{Y}_Z)_n:= \widehat{Y^{n+1}}_Z$. For any $B \in \CAlg(\Cx)$, we thus have
\[
 (\tilde{Y}_Z)_n(B) \cong Y(B)^{n+1}\by_{Y(B^{\red})^{n+1}}Z(B^{\red}),
\]
so $\tilde{Y}_Z(B)$ is the \v Cech nerve of $\pi \co Y(B)\by_{Y(B^{\red})}Z(B^{\red}) \to Z(B^{\red}) $, and thus weakly equivalent to the image of $\pi$. Since $Y$ is smooth, $Y(B) \to Y(B^{\red})$  is surjective, so $\tilde{Y}_Z(B) \to Z(B^{\red})$ is a weak equivalence.

We therefore have a weak equivalence $N\Q.\tilde{Y}_Z \to (\Q.Z)_{\dR}$ in $\Ch_{\Q}(\Aff_{\Cx})$, with $ N\Q.\tilde{Y}_Z$ cofibrant, so
\[
 \oR\C^{\infty}(Z_{\dR},\R) \simeq N_c \C^{\infty}(\tilde{Y}_Z, \R),
\]
where $N_c$ denotes cosimplicial conormalisation. The construction of \cite{Gr} applied to the pro-ring $ \C^{\infty}(\tilde{Y}_Z)$ gives a quasi-isomorphism from $N_c \C^{\infty}(\tilde{Y}_Z, \R)$ to $\widehat{A^{\bt}(Y(\Cx),\R)}_{Z}$, the completion of the real  de Rham complex of $Y(\Cx)$ with respect to $\ker(\C^{\infty}(Y,\R) \to \C^{\infty}(Z,\R)$.  If $I$ is the kernel of $\C^{\infty}(\tilde{Y}_Z, \R) \to \C^{\infty}(\hat{Y}_Z, \R)$, then filtration by powers of $I$ maps quasi-isomorphically to the brutal truncation filtration on $\widehat{A^{\bt}(Y(\Cx),\R)}_{Z}$.

Writing $\widehat{\sA^{\bt}(Y(\Cx),\R)}_{Z} $ for the natural complex of sheaves on $Y(\Cx)_{\an}$ whose global sections are $\widehat{A^{\bt}(Y(\Cx),\R)}_{Z}$, the proof of \cite[Theorem IV.1.1]{hartshorneDRCoho} generalises from the holomorphic setting to the smooth setting to show that $\R_{Z(\Cx)} \to   \widehat{\sA^{\bt}(Y(\Cx),\R)}_{Z} $ is a sheaf quasi-isomorphism on $Y(\Cx)_{\an}$. Because the sheaves $\widehat{\sA^n(Y(\Cx),\R)}_{Z}$ are flabby, we thus have 
\[
  \oR\Gamma(Y(\Cx)_{\an}, \R_{Z}) \simeq \widehat{A^{\bt}(Y(\Cx),\R)}_{Z};
\]
since $\oR\Gamma(Z(\Cx)_{\an}, \R) \simeq \oR\Gamma(Y(\Cx)_{\an}, \R_{Z})$, this gives
\[
 \oR\Gamma(Z(\Cx)_{\an}, \R) \simeq \oR\C^{\infty}(Z_{\dR},\R).
\]

In order to make this quasi-isomorphism natural in $Z$, we would need to choose $Y$ functorially. However, passing to filtered colimits allows the construction to work for any pro-smooth affine scheme $Y$, and we can just take $Y=\Spec \Symm_{\Cx}O(Z)$. 
\end{proof}

\begin{corollary}\label{blanccor0}
For any spectrum-valued presheaf $F$ on $\Aff_{\Cx}$ which is locally of finite presentation, we have
\[
 \oR\C^{\infty}(F_{\dR},\R) \simeq \HHom_{\R} ((|F|_{\bS})_{\R}, \R),
\]
 for Blanc's functor $|-|_{\bS}$ from \cite[Definition 3.13]{blancToplKTh}.
\end{corollary}
\begin{proof}
The functor $|-|_{\bS}$ is defined as the derived enriched left Kan extension of the singular space functor $\ssp \co \Aff_{\Cx} \to s\Set$ along the Yoneda embedding of $\Aff_{\Cx}$ in spectral presheaves. Since $F$ is locally of finite presentation, we can calculate $|F|_{\bS}$  as the enriched left Kan extension along the Yoneda embedding of $\Aff_{\Cx, \fin}$, either by the reasoning of Lemma \ref{indlemma} or because the comma category $ (\Aff_{\Cx}\da F)$ is equivalent to $\pro(\Aff_{\Cx ,\fin}\da F)$, so has $(\Aff_{\Cx ,\fin}\da F)$ as a final subcategory.
 
Now, observe that $F \mapsto F_{\dR}$ is the enriched derived left Kan extension of $Z \mapsto Z_{\dR}$ along the Yoneda embedding, so Proposition \ref{dRprop} combines with the definition of $ \oR\C^{\infty}$ to show that
\begin{align*}
 \oR\C^{\infty}(F_{\dR},\R) &\simeq \int^h_{Z \in \Aff_{\Cx, \fin}} \HHom_{\R}(F(Z)_{\R}, \oR\C^{\infty}(Z_{\dR},\R))\\
&\simeq \int^h_{Z \in \Aff_{\Cx, \fin}} \HHom_{\R}(F(Z)_{\R}\ten_{\R} N\CC_{\bt}(\ssp(Z),\R), \R)\\
&\simeq \HHom_{\Q}((|F|_{\bS})_{\R},\R),
\end{align*}
where $\int^h$ denotes homotopy end in the category of real chain complexes.
 \end{proof}

\begin{corollary}\label{blanccor}
For any complex dg category $\cA$, there is a canonical zigzag
\[
  \oR\C^{\infty}(\uline{K}(\cA)_{\dR},\R)\simeq \HHom_{\R}(K^{\Blanc,st}(\cA)_{\R},\R)
\]
of quasi-isomorphisms, where $K^{\Blanc,st}(\cA)$ is Blanc's semi-topological $K$-theory. 
\end{corollary}
\begin{proof}
 The $K$-theory presheaf $\uline{K}(\cA)$ is locally of finite presentation, and by \cite[Definition 4.1]{blancToplKTh}, $|\uline{K}(\cA)|_{\bS}=K^{\Blanc,st}(\cA)$.
\end{proof}

\begin{remark}
 Beware that Blanc's semi-topological $K$-theory is not the same as that of Friedlander and Walker. However, it shares the property \cite[Definition 4.13 and Proposition 4.32]{blancToplKTh} that  inverting the Bott element $\beta$ gives a weak equivalence between $K^{\Blanc,st}(X)[\beta^{-1}]$ and the topological $K$-theory of $X(\Cx)_{\an}$ for any separated $\Cx$-scheme of finite type.
\end{remark}

\begin{remark}\label{KHrmk}
For any affine $\Cx$-scheme $Z$, a model for the associated simplicial set $\ssp(Z)$ is given by $n \mapsto \C^{\infty}(\Delta_{\R}^n, Z(\Cx))= Z(\C^{\infty}(\Delta_{\R}^n, \Cx))$, where 
we write $\Delta_{\R}^n$ for the hyperplane $\{\sum x_i=1\}$ in $\R^{n+1}$. In other words, $\ssp(Z)\simeq \ho\LLim_{n \in \Delta^{\op}} Z(\C^{\infty}(\Delta_{\R}^n, \Cx))$. For any spectral presheaf $F$ on $\Aff_{\Cx}$, we thus have
\[
|F|_{\bS} \simeq   \ho\LLim_{n \in \Delta^{\op}} F(\C^{\infty}(\Delta_{\R}^n, \Cx))
\]
by left Kan extension, and in particular
\[
 K^{\Blanc,st}(\cA)\simeq \ho\LLim_{n \in \Delta^{\op}} K(\cA\ten_{\R}\C^{\infty}(\Delta_{\R}^n, \R)).
\]

There is thus a close analogy between  Blanc's semi-topological $K$-theory and Weibel's homotopy $K$-theory \cite{weibelKH}, the latter being
\[
KH(\cA)\simeq \ho\LLim_{n \in \Delta^{\op}} K(\cA\ten_{\Z}\Z[t_0, \ldots, t_n]/(1-\sum t_i)). 
\]
Thus $K^{\Blanc,st}$ should perhaps be thought of as the stabilisation of $K$ with respect to homotopies for the manifold $\R$, in the same way that $KH$ is the stabilisation of $K$ with respect to $\bA^1$-homotopies.
\end{remark}

\section{Cyclic homology and Goodwillie's comparison}

\subsection{Goodwillie's comparison}

\begin{definition}
Given a commutative ring $k$, write  $d\Alg_k$ for the category of associative, not necessarily commutative dg algebras concentrated in non-negative chain degrees. Given $A \in d\Alg_k$, write $d\Alg_A$ for the comma category $(A \da d\Alg_k)$.
\end{definition}

\begin{definition}
 Given a commutative ring $k$ and a dg category $\cA$ over $k$, write $\HC^k(\cA)$ for the chain complex associated to  cyclic homology  of $\cA$ over $k$. 
\end{definition}

\begin{definition}
 Given a functor $F$ from dg categories to chain complexes, and a dg functor $f \co \cA \to \cB$, write $F(f)$ for the homotopy fibre of $F(\cA) \to F(\cB)$.
\end{definition}

\begin{theorem}\label{goodwillie}
If $f \co A \to B$ is a morphism in $ d\Alg_{\Z}$ for which $\H_0A \to \H_0B$ is a surjection with nilpotent kernel, then there is a canonical zigzag $K(f)_{\Q} \simeq \HC^{\Q}(f\ten_{\Z}\Q)_{[-1]}$ of weak equivalences between  chain complexes given by the Chern character. 
\end{theorem}
\begin{proof}
 Using the Quillen equivalence between simplicial algebras and dg algebras, the main theorem of \cite{goodwillieChern} gives $ \tilde{K}(f)_{\Q} \simeq \HC^{\Q}(f\ten_{\Z}\Q)_{[-1]}$ for connective $K$-theory $\tilde{K}$.

Since the morphisms 
$A[t] \to B[t]$ and $A[t,t^{-1}] \to B[t,t^{-1}]$
are also nilpotent surjections on $\H_0$, iterated Bass delooping then extends the comparison to non-connective $K$-theory via 
 \cite[Theorem 7.1]{schlichtingNegative}.

More explicitly, a resolution of the good truncation $(\tau_{\ge n-1}K(f)_{\Q})_{[-1]}$ is given by the total complex of
\[
 \tau_{\ge n}K(f)_{\Q} \to  \tau_{\ge n}K(f[t])_{\Q}\oplus \tau_{\ge n}K(f[t^{-1})_{\Q}\to \tau_{\ge n}K(f[t,t^{-1}])_{\Q},
\]
 and $\tilde{K}= \tau_{\ge 0}K$. Meanwhile, the total complex of 
\[
 \HC^{\Q}(f\ten_{\Z}\Q) \to \HC^{\Q}(f\ten_{\Z}\Q[t]) \oplus \HC^{\Q}(f\ten_{\Z}\Q[t^{-1}])\to \HC^{\Q}(f[t,t^{-1}]\ten_{\Z}\Q)
\]
is quasi-isomorphic to $\HC^{\Q}(f\ten_{\Z}\Q)_{[-1]}$, because cohomology of projective space gives $\HC^{\Q}(f\ten_{\Z}\bP^1_{\Q})\simeq \HC^{\Q}(f\ten_{\Z}\Q) \oplus \HC^{\Q}(f\ten_{\Z}\Q)$. Inductively comparing these total complexes then gives
\[
 \tau_{\ge n}K(f)_{\Q} \simeq \HC^{\Q}(f\ten_{\Z}\Q)_{[-1]}
\]
 for all $n \le 0$. 
\end{proof}

\begin{corollary}\label{goodwilliecor0}
 For any $A \in d\Alg_{\Cx}$, there is a canonical zigzag of weak equivalences 
\[
 \cone(\uline{K}(A)_{\Q} \to \uline{K}(A)_{\dR,\Q}) \simeq \cone(\uline{\HC}^{\Q}(A) \to \uline{\HC}^{\Q}(A)_{\dR})_{[-1]}
\]
of presheaves in $\Ch_{\Q}(\Aff_{\Cx})$ given by the Chern character.
\end{corollary}
\begin{proof}
 For any $B \in \CAlg(\Cx)_{\fin}$, the morphism $B \to B^{\red}$ has nilpotent kernel. Thus $A\ten_{\Cx}B \to A\ten_{\Cx}B^{\red}$ satisfies the conditions of  Theorem \ref{goodwillie}, giving
\[
 \cone(\uline{K}(A)(B)_{\Q} \to \uline{K}(A)_{\dR}(B)_{\Q}) \simeq \cone(\uline{\HC}^{\Q}(A)(B) \to \uline{\HC}^{\Q}(A)_{\dR}(B))_{[-1]}.
\]
The result then extends to the whole of $ \CAlg(\Cx)$ by passing to filtered colimits.
\end{proof}

Zariski descent allows us to extend this far more generally: 
\begin{corollary}\label{goodwilliecor}
 If $X$ is a quasi-compact semi-separated derived scheme  over $\Cx$, with 
$\sE \in d\Alg_{\sO_X}$ perfect as an $\sO_X$-module,  and $\cA$ is a semi-orthogonal summand of $\per_{dg}(\sE)$, the Chern character gives
 \[
 \cone(\uline{K}(\cA)_{\Q} \to \uline{K}(\cA)_{\dR,\Q}) \simeq \cone(\uline{\HC}^{\Q}(\cA) \to \uline{\HC}^{\Q}(\cA)_{\dR})_{[-1]}.
\]
\end{corollary}
\begin{proof}
A  derived scheme $X$ is a derived $n$-geometric stack 
whose underived truncation $\pi^0X$ is a scheme. The quasi-compactness  hypothesis for $\pi^0X$  means that there is a Zariski cover $U_1, \ldots, U_n$ for $X$ by derived affines, so each $O(U_i)$ is a commutative object of $d\Alg_{\Cx}$.   Now take the reduced \v Cech resolution $\check{X}$ for $X$, given by
\[
 \ldots \coprod_{i<j<k} U_i\by^h_XU_j\by   \Rrightarrow  \coprod_{i<j} U_i\by^h_XU_j  \implies \coprod_i U_i;
\]
this diagram terminates at $U_1\by^h_XU_2\by^h_X \ldots \by^h_X U_n $ and  since $X$ is semi-separated, 
it is a diagram of derived affines.

Zariski descent for $K$ then gives a resolution of  $\uline{K}(\per_{dg}(\sE) )$ by 
\[
 \uline{K}(\sE(\check{X}_0 )) \to  \uline{K}(\sE(\check{X}_1 )) \to \ldots \to \uline{K}(\sE(\check{X}_n )),
\]
so Corollary \ref{goodwilliecor0} implies the required result when $\cA=  \per_{dg}(\sE)$.

Finally, for a  $\cA$  a semi-orthogonal summand of $\per_{dg}(\sE)$ as in \cite[\S 1.2]{orlovSmoothProperNC}, additivity of $K$  gives $\uline{K}(\cA)_{\Q}$ as a summand of $\uline{K}(\per_{dg}(\sE))$; likewise, additivity of $\HC$ gives $\uline{\HC}^{\Q}(\cA)$ as a summand of $\uline{\HC}^{\Q}(\per_{dg}(\sE)) $, so the  equivalence extends to $\cA$. 
\end{proof}

\begin{remarks}\label{orlovrmk}
 Beware that the hypothesis $\sE \in d\Alg_{\sO_X}$ implies that $\sH_i(\sE)=0$ for $i<0$, so is stronger than the customary cohomological boundedness condition for derived non-commutative schemes  in \cite[Definition 3.3]{orlovSmoothProperNC}. However, Corollary \ref{goodwilliecor} applies to all quasi-compact semi-separated derived schemes $X$ (taking $\sE=\sO_X$) and to all of Orlov's geometric noncommutative schemes of \cite[Definition 4.3]{orlovSmoothProperNC}.

The conclusion of Corollary \ref{goodwilliecor} will also hold for any complex dg category $\cA$ Morita equivalent to one satisfying the conditions, and to any filtered colimit of such. This applies for instance if  $\cA$ is a dg category with $\H_i\cA(x,y)=0$ for all $i<0$ and all objects $x,y$;  it is a filtered colimit of its dg subcategories $\cA|_S$ on finite subsets $S$ of  objects, and  each $\cA|_S$ is Morita equivalent to $\tau_{\ge 0}(\bigoplus_{x,y \in S} \cA(x,y)) \in d\Alg_{\Cx}$. By Morita invariance, the conclusions thus hold for any complex dg category $\cA$ for which $\per_{dg}(\cA)$ has a set of  generators with no positive $\Ext$-groups between them; this amounts to having what Bondarko calls a weight structure (whereas a $t$-structure means having  no negative  $\Ext$-groups between generators). 
\end{remarks}

\subsection{Smooth functions on smash products of Eilenberg--MacLane spectra}\label{eilmacsn}

\begin{proposition}\label{keyprop}
Given  a cosimplicial real vector space $V$ with $\H^0V=0$, finite-dimensional in each level, the morphism 
\[
 \Symm_{\R}V \to \C^{\infty}(V^{\vee}, \R)
\]
of cosimplicial vector spaces is a quasi-isomorphism.
\end{proposition}
\begin{proof}

Let $Q$ be the quotient of $\Symm_{\R}V \to \C^{\infty}(V^{\vee}, \R)$; we wish to show that this is acyclic. Now consider the algebraic and smooth de Rham complexes respectively, giving morphisms
\[
 \Omega^{\bt}_{\R}(\Symm_{\R}V) \to A^{\bt}(V^{\vee}, \R)
\]
of cosimplicial cochain complexes. Each of these cochain complexes is a resolution of $\R$, so the quotient $T$ is levelwise acyclic as a cosimplicial cochain complex, and in particular $\H^*(\Tot T)=0$.

Taking the  filtration on $T$ given by brutal truncations in the cochain direction, we get a convergent spectral sequence
\[
 \H^i(Q \ten \L^j V) \abuts \H^{i+j}(\Tot T)=0.
\]
If $Q$ is not acyclic, let $m$ be least such that $\H^m(Q)\ne 0$.
Since $\H^0V=0$, the Eilenberg--Zilber theorem \cite[Theorem 8.5.1]{W} gives $\H^{< m+j}(Q \ten \L^jV)=0$;  therefore we must have $\H^m(\Tot T) \ne 0$, which is a contradiction. 
\end{proof}

\begin{definition}
 Given $V \in \Ch_{\Cx}$, define $\uline{V} \in \Ch_{\Q}(\Aff_{\Cx})$ by $\uline{V}:= V\ten_{\Cx}\bG_a$. Explicitly, this presheaf is given by $\uline{V}(Z):= V\ten_{\Cx}O(Z)$.
\end{definition}

\begin{definition}\label{Hdef}
For any chain complex $V$, we may form the Eilenberg--MacLane spectrum $HV$ by $(HV)_n := N^{-1}\tau_{\ge 0}(V_{[-n]})$, where $N^{-1}$ is Dold--Kan denormalisation. For $V \in \Ch_{\Cx}$, we therefore obtain a spectrum $\uline{HV}$ in simplicial ind-affine $\Cx$-schemes, which we can regard as a spectral presheaf on $\Aff_{\Cx}$.
\end{definition}

\begin{corollary}\label{keycor}
For $V^{(1)}, \ldots, V^{(r)} \in \Ch_{\Cx}$, there is a canonical zigzag of quasi-isomorphisms
\[
\oR\C^{\infty}(\uline{HV}^{(1)} \wedge\ldots \wedge \uline{HV}^{(r)}, \R)  \simeq \HHom_{\R}(V^{(1)}\ten_{\R} \ldots \ten_{\R} V^{(r)},\R)
\]
of real chain complexes.
\end{corollary}
\begin{proof}
Since $\oR\C^{\infty}(-, \R)$ sends filtered colimits to filtered limits,  we may reduce to the case where the complexes  $V^{(1)}, \ldots, V^{(r)}$ all have finite total dimension. Now, the spectra $ HV$ and $H'V:= \{N^{-1}\tau_{\ge 1}(V_{[-n]})\}_n $ are stably equivalent, and the cosimplicial vector spaces $(H'V^{(1)} \by \ldots \by H'V^{(r)})_n^{\vee}$ all satisfy the conditions of Proposition \ref{keyprop}. 

Write 
\[
 \bar{\C}^{\infty}(X_1 \wedge\ldots \wedge X_r,\R):= \ker(\C^{\infty}(X_1 \by\ldots \by X_r,\R)\to \prod_i \C^{\infty}(X_1 \by\ldots  \by \{x_i\} \by \ldots \by X_r,\R))
\]
for pointed manifolds $(X_i, x_i)$, and note that this has a cochain resolution by products of vector spaces of the form $\C^{\infty}( (\prod_{j \in S} X_j)  \by \prod_{i \notin S} \{x_i\},\R))$ for subsets $S \subset [1,n]$.

  Substituting in Definition \ref{RCinftydef} then gives
\begin{align*}
 &\oR\C^{\infty}(\uline{HV}^{(1)} \wedge\ldots \wedge \uline{HV}^{(r)}, \R)\\
 &\simeq \Lim_n \bar{\C}^{\infty}( (HV^{(1)})_n \wedge\ldots \wedge (HV^{(r)})_n, \R)_{[-nr]}, \\
&\simeq \Lim_n (\Symm^+_{\R} \tau_{\le 0}(V^{(1)}_{[-n]})^{\vee})
\ten_{\R} \ldots \ten_{\R} (\Symm^+_{\R} \tau_{\le 0}(V^{(r)}_{[-n]})^{\vee})_{[-nr]},\\
&\simeq  V(1)^{\vee}\ten_{\R} \ldots \ten_{\R} V(r)^{\vee},\\
&\simeq  \HHom_{\R}(V(1)\ten_{\R} \ldots \ten_{\R} V(r),\R),
\end{align*}
where $\Symm^+V = \bigoplus_{i>0} \Symm^iV$.
 \end{proof}

\begin{corollary}\label{keycor2}
 For $V^{(1)}, \ldots, V^{(r)} \in \Ch_{\Cx}$, there is a canonical zigzag of quasi-isomorphisms
\[
 \oL \Sm^* ( \uline{V}^{(1)}\ten_{\Q} \ldots \ten_{\Q}\uline{V}^{(r)})^{\vee} \simeq \HHom_{\R}(V^{(1)}\ten_{\R} \ldots \ten_{\R} V^{(r)},\R)
\]
of real chain complexes.
\end{corollary}
\begin{proof}
The spectral presheaf $\uline{HV}$ on $\Aff_{\Cx}$ has the important property that $\uline{HV}_{\Q}\simeq \uline{V}$; moreover we have
\[
(\uline{HV}^{(1)} \wedge\ldots \wedge \uline{HV}^{(r)})_{\Q} \simeq \uline{V}^{(1)}\ten_{\Q} \ldots \ten_{\Q}\uline{V}^{(r)},
\]
so
\begin{align*}
 \oL \Sm^* ( \uline{V}^{(1)}\ten_{\Q} \ldots \ten_{\Q}\uline{V}^{(r)})^{\vee} &\simeq \oR\C^{\infty}(\uline{HV}^{(1)} \wedge\ldots \wedge \uline{HV}^{(r)}, \R),\\
&\simeq \HHom_{\R}(V^{(1)}\ten_{\R} \ldots \ten_{\R} V^{(r)},\R).
\end{align*}
\end{proof}

We will also need to look at smooth functions on de Rham presheaves:

\begin{proposition}\label{keydRprop}
  For $V^{(1)}, \ldots, V^{(r)} \in \Ch_{\Cx}$, there are  canonical zigzags of quasi-isomorphisms
\[
 \oL \Sm^* ( \uline{V}^{(1)}\ten_{\Q} \ldots \ten_{\Q}\uline{V}^{(r)})_{\dR}^{\vee} \simeq \oR\C^{\infty}((\uline{HV}^{(1)} \wedge\ldots \wedge \uline{HV}^{(r)})_{\dR}, \R) \simeq 0
\]
of real chain complexes.
\end{proposition}
\begin{proof}
We have  $(HA^{(1)} \wedge\ldots \wedge HA^{(r)})_{\Q} \simeq A^{(1)}\ten_{\Z}\ldots \ten_{\Z}A^{(r)}\ten_{\Z}\Q$ for any chain complexes $A^{(i)}$ of abelian groups, which implies the
  first equivalence. As in the proof of Corollary \ref{keycor}, the second equivalence reduces to calculating spaces of the form $\bar{\C}^{\infty}((\uline{U}^{(1)} \by\ldots \by \uline{U}^{(r)})_{\dR}, \R)$ for finite-dimensional complex vector spaces $U^{(i)}$. But by Proposition \ref{dRprop}, these are all quasi-isomorphic to $\R$, so 
 $\bar{\C}^{\infty}((\uline{U}^{(1)} \wedge\ldots \wedge \uline{U}^{(r)})_{\dR}, \R)\simeq 0$.
\end{proof}

\subsection{Smooth functions on cyclic homology}

\begin{theorem}\label{HCthm}
 Given a  dg category $\cA$ over $\Cx$, there are  canonical zigzags of quasi-isomorphisms
\begin{align*}
\oL\Sm^*(\uline{HC}^{\Q}(\cA))^{\vee} &\simeq \HHom_{\R}(\HC^{\Cx}(\cA),\R)\\
\oL\Sm^*(\uline{HC}^{\Q}(\cA)_{\dR})^{\vee} &\simeq 0.
\end{align*}
of real chain complexes.
\end{theorem}
\begin{proof}
Since $\Cx$ is \'etale over $\R$, the map $ \HC^{\R}(\cA) \to\HC^{\Cx}(\cA)$ is a quasi-isomorphism.
 Because $\uline{\HC}^{\Q}(\cA)$ is formed from homotopy colimits of rational tensor products of the complexes $\uline{\cA(x,y)}$ for objects $x, y \in \cA$, Corollary \ref{keycor2} immediately then implies the first equivalence, and Proposition \ref{keydRprop} the second.
\end{proof}

\begin{corollary}\label{maincor}
  If $X$ is  a quasi-compact semi-separated derived scheme over $\Cx$, with
$\sE \in d\Alg_{\sO_X}$ perfect as an $\sO_X$-module,  and $\cA$ is a semi-orthogonal summand of $\per_{dg}(\sE)$, then
the complex $\oR\C^{\infty}(\uline{K}(\cA),\R)$ of smooth functions is given by the cone of the Chern character
\[
 \HHom_{\R}(\HC^{\Cx}(\cA)_{[-2]},\R) \xra{\ch^*} \HHom_{\R}(K^{\Blanc,st}(\cA)_{\R},\R), 
\] 
so
\[
 \oR\C^{\infty}(\uline{K}(\cA),\R) \simeq \HHom_{\R}(  \HN^{\Cx}(\cA)\by^h_{ \HP^{\Cx}(\cA)}K^{\Blanc,st}(\cA)_{\R}, \R).
\]

\end{corollary}
\begin{proof}
 The first statement just combines Corollaries \ref{blanccor} and \ref{goodwilliecor} with Theorem \ref{HCthm}. The second statement then follows because $\HC_{[-2]}$ is the cone of the map $\HN \to \HP$ from negative cyclic homology to periodic cyclic homology.
\end{proof}

See Remarks \ref{orlovrmk} for examples satisfying these hypotheses.

\begin{remark}\label{Delignecoho}
The expression $\HN^{\Cx}(\cA)\by^h_{ \HP^{\Cx}(\cA)}K^{\Blanc,st}(\cA)_{\R}$ (or equivalently the homotopy fibre of  $\ch \co K^{\Blanc,st}(\cA)_{\R} \to \HC^{\Cx}(\cA)_{[-2]}$)  is very closely related  to real Deligne cohomology. If $X$ is a separated $\Cx$-scheme of finite type, then $K^{\Blanc, top}(X):= K^{\Blanc,st}(X)[\beta^{-1}]$ is the topological $K$-theory of $X(\Cx)_{\an}$. 

For a smooth separated $\Cx$-scheme $X$, the HKR isomorphism then identifies $\HN^{\Cx}(X)\by^h_{ \HP^{\Cx}(\cA)}(K^{\Blanc,st}(\cA)_{\R}[\beta^{-1}])$ with
\[
 \prod_{p \in \Z} \cocone(\oR\Gamma(X(\Cx)_{\an}, \R(2\pi i)^p) \xra{\ch} \oR\Gamma(X, \sO_X \to \ldots \to \Omega^{p-1}_{X/\Cx}))^{[2p]}  
\]
identifying chain and cochain complexes in the obvious way;
when $X$ is smooth and proper, this is just real Deligne cohomology $\prod_{p \in \Z}\oR\Gamma_{\cD}(X, \R(p))^{[2p]}$.

There does not yet seem to be a simple description of $K^{\Blanc,st}(X)$, but for smooth affines it is connective by \cite[Theorem 4.6]{blancToplKTh}, and for products of projective spaces it is the connective truncation of topological $K$-theory by \cite[Theorems 4.5 and 4.6]{blancToplKTh} and \cite[Theorem 7.3]{ThomasonTrobaugh}. Taking Gysin sequences, this means that for any complement $X$ of a  product of projective spaces by a disjoint union of products of projective spaces (the simplest interesting examples are $X= \bA^1, \bG_m$), we have $ K^{\Blanc,st}(X)_{\Q} \simeq  \prod_{p\ge 0} W_0\oR\Gamma(X(\Cx)_{\an},\Q(p))[2p]$, for Deligne's weight filtration $W$.
\end{remark}

\begin{remark}\label{regulator}
For $Z \in \Aff_{\Cx}$, there is a natural morphism $\C^{\infty}(Z, \R) \to \R^{Z(\Cx)}$ from smooth functions to discontinuous functions. The spectral derived left Kan extension of $ Z \mapsto (\R^{Z(\Cx)})^{\vee}$ is just given by $F \mapsto \HHom_{\R}(F(\Spec \Cx)_{\R}, \R)^{\vee}$, so there is a natural map $\oR \C^{\infty}(F, \R) \to \HHom_{\R}(F_{\R}(*), \R) $. Similarly, we have $\oR O^{\infty}(F) \to \HHom_{\R}(F_{\R}(*), \R)$ for spectral presheaves on $\C^{\infty}\Aff$. 

Applied to $\uline{K}(\cA)$, these give $\oR \C^{\infty}(\uline{K}(\cA) , \R) \to \HHom_{\R}(K(\cA)_{\R}, \R)$, which recovers Beilinson's regulator \cite{beilinson} on duals when applied in the context of Remark \ref{Delignecoho}.
\end{remark}


\subsection{Related constructions and generalisations}\label{vistas}

\subsubsection{Smaller diagram categories}\label{smallerdiagrams}

As observed in \S \ref{derivedcompletionsn}, the calculation of $\oR\C^{\infty}(\uline{K}(\cA), \R)$ can be made by looking at the restricted presheaf $\uline{K}(\cA)|_{\C}$:
\[
 \oR\C^{\infty}(\uline{K}(\cA), \R)\simeq \oR\C^{\infty}(\uline{K}(\cA)|_{\C}, \R),
\]
 taking $\C$ to be any of the categories $\Aff_{\Cx}$, $\Aff_{\Cx, \fin}, \C^{\infty}\Aff, \C^{\infty}\Aff_{\fin}$. The same is true for $K^{\Blanc, st}(\cA)$, which  \cite[Theorem  3.18]{blancToplKTh} shows can also be recovered by restricting to the category of smooth complex affines.  
By Remark \ref{KHrmk} we can even recover $K^{\Blanc, st}(\cA)$ from its restriction to the simplex category  $\{\Delta^n_{\R} \cong \R^n\}_n$ living inside $\C^{\infty}\Aff_{\fin}$. 

However, calculations involving $\uline{K}(\cA)_{\dR}$ only make sense in a  category having  non-reduced objects. The smallest choice of category $\C$ for which Corollary \ref{blanccor}, Theorem \ref{HCthm} and Corollary \ref{maincor} would still hold is probably the full subcategory of $\ind(\C^{\infty}\Aff)$ consisting of objects of the form $\R^m \by \widehat{\R^n}_0$, with associated pro-$\C^{\infty}$-ring $\C^{\infty}(\R^m,\R)[[ t_1, \ldots, t_n]]$. Theorem \ref{HCthm} still holds because $(\R^m)_{\dR}$ has a simplicial resolution by such spaces. Corollary \ref{blanccor} still holds because  $\Hom_{\C}(\Delta_{\R}^{\bt},-)$ is a resolution of the constant copresheaf $*$ on $\C$, while $\C^{\infty}((-)_{\dR},V)$ is quasi-isomorphic to the constant presheaf $V$.


\subsubsection{Algebraic functions on algebraic $K$-theory}

Instead of looking at the Kan extension of smooth functions $Z \mapsto \C^{\infty}(Z, \R)$ applied to $\uline{K}(\cA)$, we could look at the Kan extension of algebraic functions $Z \mapsto \Gamma(Z, \sO)$. The proof of Corollary \ref{blanccor} adapts to give  $\oR\Gamma(\uline{K}(\cA)_{\dR},\sO)\simeq \HHom_{\Cx}(K^{\Blanc,st}(\cA)_{\Cx},\Cx)$, while Theorem \ref{HCthm} adapts without needing to appeal to Proposition \ref{keyprop}. The analogue of Corollary \ref{maincor} then gives $ \oR\Gamma(\uline{K}(\cA)_{\dR},\sO)$ dual to the homotopy fibre of $ \ch \co K^{\Blanc,st}(\cA)_{\Cx} \to \HC^{\Cx}(\cA)_{[-2]}$.


\subsubsection{Refinements for non-proper dg categories}
 
As in \S \ref{derivedcompletionsn}, the smooth functions $\oR\C^{\infty}(\uline{K}(\cA),\R)$ are recovered from the presheaf $\uline{K}(\cA)|_{\C^{\infty}\Aff}$. This presheaf is only really a sensible object of study for proper dg categories, because for manifolds $Z$, the map 
\[
 \C^{\infty}(Z,\R) \ten_{\R} \cA(x,y) \to \C^{\infty}(Z, \cA(x,y))
\]
 only tends to be a quasi-isomorphism when $\cA(x,y)$ is a perfect complex.
 
When cyclic homology is finite-dimensional, it can be recovered from Theorem \ref{HCthm} by taking duals.
If we wanted to refine the comparison of Theorem \ref{HCthm}  to recover cyclic homology  more generally, we would instead have to work with ind-objects,    looking at the left Kan extension of the composition $(O^{\infty})^{\vee} \co \C^{\infty}\Aff \to \Ch(\pro(\FD\Vect_{\R})) \to \ind(\Ch(\pro(\FD\Vect_{\R})))$. 

A slightly more natural choice would be the left Kan extension of functor 
$(O^{\infty})^{\vee} \co \ind( \C^{\infty}\Aff) \to \ind(\Ch(\pro(\FD\Vect_{\R})))$,  applied to the presheaf
\[
\uline{K}(\cA)( \Lim_{\alpha} B_{\alpha}) := K(\ho\Lim_{\alpha} \cA\ten_{\R}B_{\alpha}) 
\]
 on $\ind(\C^{\infty}\Aff)$.

However, since Deligne cohomology for non-proper schemes is not defined in terms of their cyclic homology, there is less motivation for seeking such a refinement.


\subsubsection{$K$-theory of  dg categories enriched in topological vector spaces}

In \cite{smoothK}, we considered certain Fr\'echet algebras and their associated algebraic $K$-theory presheaves on a category of Fr\'echet manifolds, and thus obtained a comparison between real Deligne cohomology and smooth functions on $K$-theory. The setting of this paper is inadequate to recover such a result because $\C^{\infty}$-rings only model smoothly realcompact spaces.

As in \cite{KrieglMichor}, the largest class of topological vector spaces for which it makes sense to talk about $\C^{\infty}$-morphisms consists of the convenient vector spaces.  There is as natural monoidal structure on the category of convenient vector spaces, given by $c^{\infty}$-completions $\hten$ of bornological tensor products --- by \cite[Proposition 5.8]{KrieglMichor}, these operations coincide with projective tensor products on the subcategory of Fr\'echet spaces. 

 For any category $\cA$ enriched in chain complexes of convenient $\Cx$-vector spaces, and any infinite-dimensional manifold $M$ modelled on convenient vector spaces, there is then  a dg category $\C^{\infty}(M, \cA)$ over $\Cx$.

\medskip
To adapt the methods of this paper to such a setting, we need to incorporate non-reduced objects in order to define meaningful de Rham presheaves. The obvious generalisation of $\C^{\infty}$-rings might be the convenient coalgebras of \cite[23.13]{KrieglMichor}, but these do not seem to have a suitable analogue of the nilradical ideal. Instead, we can take a lead from \S \ref{smallerdiagrams} and just consider formal completions of manifolds. Let $\cJ$ be the category of products $T \by \hat{V}_0$ for convenient vector spaces $T,V$, with morphisms
\begin{align*}
 \Hom_{\cJ}(S \by \hat{U}_0, T \by \hat{V}_0)&= \C^{\infty}(S,J^{\infty}_0(U,T) \by\Poly^{\infty}(U,V))\\ 
&=\C^{\infty}(S,J^{\infty}_0(U,T) \by J^{\infty}_0(U,V)_0),
\end{align*}
for the jet spaces $J^{\infty}$ of \cite[41.1]{KrieglMichor}. The following reasoning will apply equally well if we allow $T$ to be a star-shaped $c^{\infty}$-open in a convenient vector space.

Given a convenient dg category $\cA$, we  have a presheaf $\uline{K}(\cA)$ defined on $\cJ$ by
\begin{align*}
 &\uline{K}(\cA)(T \by \hat{V}_0):= \ho\Lim_n K(\C^{\infty}(T,J^n_0(V,\cA))), \\
 \text{with }\quad &\uline{K}(\cA)_{\dR}(T \by \hat{V}_0):= \uline{K}(\cA)(T)=  K(\C^{\infty}(T,\cA)).
\end{align*}

\medskip
The Poincar\'e lemma of \cite[33.20]{KrieglMichor} allows the proofs of Corollary \ref{blanccor} and Theorem \ref{HCthm} to adapt to this context, replacing tensor products $\ten_{\R}$ with their bornological $c^{\infty}$-completions $\hten$. 
Moreover, the analogues of Corollary \ref{blanccor} and Theorem \ref{HCthm} hold for smooth functions $\C^{\infty}(-,V)$ with coefficients in any convenient vector space $V$. 

As in \cite[23.13]{KrieglMichor}, there is a free convenient vector space functor $\lambda$ from Fr\"olicher spaces to convenient vector spaces, with the property that the space $L(\lambda X,V)$ of bounded linear maps is isomorphic to $\C^{\infty}(X,V)$. For finite-dimensional manifolds $M$, $\lambda M= \C^{\infty}(M,\R)'$, the space of compactly supported distributions on $M$. The functor $\lambda$ extends naturally to $\cJ$ by setting $ \lambda(T \by \hat{V}_0 ):= (\lambda T) \hten (\bigoplus_{i \ge 0} \widehat{\Symm}^i_{\R}V)$.

\medskip
We can then consider the derived enriched left Kan extension $\oL\lambda$ of $\lambda$ along the Yoneda embedding of $\cJ$ in spectral presheaves.

Observe that if $\cA$ is a proper dg category over $\Cx$, the maps
$
 \cA\ten_{\R}\C^{\infty}(T,J^n_0(V,\R)) \to \C^{\infty}(T,J^n_0(V,\cA))
$
are quasi-equivalences, so $\oL\lambda \uline{K}(\cA)$ can be recovered from the presheaf $\uline{K}(\cA)|_{\C^{\infty}\Aff}$ considered throughout this paper, and its complex of bounded functionals satisfies 
$
\oL\lambda \uline{K}(\cA)'= L(\oL\lambda \uline{K}(\cA), \R) \simeq \oR\C^{\infty}( \uline{K}(\cA)|_{\Aff_{\Cx}}, \R).
$

\bigskip
For all convenient dg categories $\cA$ over $\Cx$,  Corollary \ref{blanccor} and Remark \ref{KHrmk}   adapt to give
\[
 \oL\lambda \uline{K}(\cA)_{\dR} \simeq \Tot K(\C^{\infty}(\Delta_{\R}^{\bt}, \cA))_{\R},
\]
a chain homotopy equivalence of chain complexes of convenient vector spaces, where the right-hand side is given the finest $\R$-linear topology.

\medskip
The analogue of Corollary \ref{maincor} is then that for convenient dg $\Cx$-algebras $A$ concentrated in non-negative chain degrees, we have a chain homotopy equivalence
\[
 \oL\lambda \uline{K}(A) \simeq \cocone(\Tot K(\C^{\infty}(\Delta_{\R}^{\bt}, A))_{\R} \xra{\ch} \widehat{\HC^{\Cx}}(A)_{[-2]}),
\]
where $\widehat{\HC^{\Cx}}$ is defined by analogy with cyclic homology, but using  $\hten_{\Cx}$ instead of $\ten_{\Cx}$. 

\medskip
Dually to  Remark \ref{regulator}, the morphism $ T \to \lambda(T \by \hat{V}_0 )$ of sets which sends points to distributions induces a morphism
\[
K(\cA)_{\R}\to  \oL\lambda \uline{K}(\cA)
\]
on Kan extensions, which we can  think of as a generalisation of Beilinson's regulator.


\bibliographystyle{alphanum}
\bibliography{references.bib}
\end{document}